\documentclass[12pt,a4paper]{article}
\usepackage[top=3cm, bottom=3cm, left=3cm, right=3cm]{geometry}
\usepackage{amsmath,amsthm,amssymb,graphicx, multicol, array}
\usepackage{enumerate}
\usepackage{enumitem}
\usepackage{setspace}
\usepackage{xcolor}
\usepackage{mathrsfs}
\usepackage{multicol}

\usepackage[bookmarks,colorlinks,breaklinks, pagebackref]{hyperref}
\hypersetup{linkcolor=blue,citecolor=red,filecolor=blue,urlcolor=blue} 

\usepackage[hang,flushmargin]{footmisc} 
\usepackage{tabto}

\DeclareMathOperator{\li}{li}
\DeclareMathOperator{\ord}{ord}

\DeclareMathOperator{\tdiv}{div}

\DeclareMathOperator{\disc}{disc}

\newtheorem{thm}{Theorem}[section]

\newtheorem{lem}{Lemma}[section]

\newtheorem{conj}{Conjecture}[section]
\newtheorem{exa}{Example}[section]

\newtheorem{dfn}{Definition}[section]
\newtheorem{exe}{Exercise}[section]
\newtheorem{rmk}{Remark}[section]

\newcommand{\N}{\mathbb{N}}
\newcommand{\Z}{\mathbb{Z}}

\newcommand{\C}{\mathbb{C}}
\newcommand{\F}{\mathbb{F}}
\newcommand{\tP}{\mathbb{P}}

\usepackage{fancyhdr}
\title{Squarefree Values Of Polynomials}
\date{}
\author{N. A. Carella}

\begin{document}
\maketitle	
\begin{abstract} This note presents new results for the squarefree value sets of quartic polynomials over the integers. \let\thefootnote\relax\footnote{ \today \date{} 
		}
\end{abstract}

\tableofcontents

\section{Introduction } \label{S1212I}
Let $f(t)\in\Z[t]$ be a polynomial of degree $ \deg f=d$ and let $k\geq1$ be an integer. The integer value $f(n)\in\Z$ is called $k$-free if it is not divisible by a prime power $p^k$. The earliest developments in the theory of $k$-free polynomials are attributed to many authors as \cite{NT1922}, \cite{ET1931}, \cite{ET1931}, \cite{RG1933}, \cite{HC1967}, \cite{HC1968}, et alii. Let $Q_f(k,x)=\#\{n\leq x:f(n) \text{ is \textit{k}-free}\}$. The unconditional results for polynomials of degrees $\deg f\leq 3$ were established in various papers. Nagell proved a qualitative result for $k$-free values for any polynomial of degree $ \deg f\leq k$, that is, $Q_f(k,x) \to \infty$ as $x\to \infty$, confer \cite{NT1922}.  Ricci established an asymptotic formula for the same result, that is, $Q_f(k,x)=cx+o(x)$, where $c\in (0,1)$, see \cite{RG1933}. Erdos continued this development and showed that $Q_f(k,x) \to\infty$ as $x\to\infty$, where $d-1\leq k$ and $d\geq3$, see \cite{EP1953}. The asymptotic formula $Q_f(k,x)=cx+o(x)$, where $c\in (0,1)$, for the result of Erdos was obtained by by Hooley in \cite{HC1967}.  The current state of the $k$-free polynomial problem seems to be the recent work in \cite[Theorem 1]{BW2011}. The later states that $Q_f(k,x)=cx+o(x)$ whenever $d\geq3$ and $k\geq 3d/4+1/4.$ This improved an weaker result of  Nair in \cite{NM1976}.\\

In contrast, the conditional results have no restrictions on the degrees of the polynomials, the $k$-free parameter and are more general. The \textit{abc} conjecture implies the existence of infinitely many $k$-free values of any separable polynomial $f(t)\in\Z[t]$ of degree $\deg f=d\geq 1$ over the integers. The conditional results in \cite{GR1998} and other improvements and generalizations of the conditional results are basically the same as the conjectured results.

\begin{conj}\label{conj1212I.100N}  \hypertarget{conj1212I.100N} If $f(t)\in \Z[t]$ is a separable polynomial $f(t)\ne a(t)b(t)^2$ of degree $\deg f\geq 1$, then
	\begin{equation}\label{eq1212I.100N}
		\sum_{n\leq x}\mu^2(f(n)) =x\prod_{p\geq 2}\left (1-\frac{\rho_f(p^2)}{p^2} \right )+o(x)\nonumber.
	\end{equation}
\end{conj}

The density constant
\begin{equation}\label{eq1212I.120}
	c_f=\prod_{p\geq 2}\left (1-\frac{\rho_f(p^2)}{p^2} \right )
\end{equation}
depends on the properties of the polynomial $f$.\\

Some unconditional results for a related conjecture below extends the old conjecture over the integers to a conjecture over the set of primes. 

\begin{conj}\label{conj1212I.100P}  \hypertarget{conj1212I.100P}If $f(t)\in \Z[t]$ is a separable polynomial of degree $\deg f\geq 1$, then
\begin{equation}\label{eq1212I.100P}
\sum_{p\leq x}\mu^2(f(p)) =\li(x)\prod_{q\geq 2}\left (1-\frac{\rho_f(q^2)}{q^2} \right )+o(\li(x))\nonumber.
\end{equation}
\end{conj}
The cubic polynomials case is proved in \cite{HH2014}, the squarefree problem for quartic polynomials over the prime numbers seen to be partially open.\\

This note provides some new unconditional results for quartic polynomials over the integers.

\begin{thm}\label{thm1212I.100-4A}  \hypertarget{thm1212I.100-4A}The purely quartic polynomial $T^4+1\in\Z[T]$ has infinitely many squarefree values over the integer. In particular, the number of squarefree values has the asymptotic formula
	\begin{equation}\label{eq1212K.100b}
\sum_{x\leq n\leq 2x}\mu^2(n^4+1)=c_fx+O( x^{1/2+\varepsilon}),
	\end{equation}
	where
\begin{equation}\label{eq1212K.100d}
	c_f=\prod_{p\,\equiv\, 1 \bmod 8}\left( 1-\frac{4}{p^2}\right)\approx0.981003,
\end{equation}	
is a constant. 
\end{thm}

The next interesting case for the polynomial $f(t)=T^4+2$, which is an open problem discussed in \cite{EP1953} and \cite{BB2016}, is a bit more complex due to the restrictive distribution of the primes $p$ for which the congruence $T^4+2\equiv 0 \mod p$ is solvable are slightly different. 
\begin{thm}\label{thm1212I.100-4B} \hypertarget{thm1212I.100-4B} The purely quartic polynomial $T^4+2\in\Z[T]$ has infinitely many squarefree values over the integer. In particular, the number of squarefree values has the asymptotic formula
	\begin{equation}\label{eq1212K.100f}
\sum_{x\leq n\leq 2x}\mu^2(n^4+2)=	c_f	x+O( x^{1/2+\varepsilon}),
	\end{equation}
	where
\begin{equation}\label{eq1212K.100j}
c_f=\prod_{\substack{p\equiv1 \bmod 8\\p=a^2+b^2}}\left( 1-\frac{4}{p^2}\right)\prod_{\substack{p\equiv3 \bmod 8\\p=a^2+b^2}}\left( 1-\frac{2}{p^2}\right)\approx0.757159,
\end{equation}	
is a constant. 
\end{thm}

The other case considered here is relevant to the theory of modular forms, the background details are explained in \cite{MC1980}.

\begin{thm}\label{thm1212I.100-4C} \hypertarget{thm1212I.100-4C} The quartic polynomial $T^4-2T^2+2\in\Z[T]$ has infinitely many squarefree values over the integer. In particular, the number of squarefree values has the asymptotic formula
	\begin{equation}\label{eq1212K.100e}
		\sum_{x\leq n\leq 2x}\mu^2(n^4-n^2+2)=	c_f	x+O( x^{1/2+\varepsilon}),
	\end{equation}
	where
	\begin{equation}\label{eq1212K.100Cj}
		c_f=\prod_{p\equiv1 \bmod 8}\left( 1-\frac{4}{p^2}\right)\prod_{p\equiv5 \bmod 8}\left( 1-\frac{2}{p^2}\right)\approx0.963802,
	\end{equation}	
	is a constant. 
\end{thm}

This topic can be generalized in many different directions. Some contributions to the squarefree integers detection algorithms are developed in \cite{BK2015}. Various contributions for the generalization to polynomials of several variables and a short survey of the literature appears on \cite{KJ2020}.\\

The next three sections consist of preliminary results. 
The proof of  \hyperlink{thm1212I.100-4A}{Theorem} \ref{thm1212I.100-4A} appears in  \hyperlink{S5544-4A}{Section} \ref{S5544-4A}, the proof of  \hyperlink{thm1212I.100-4B}{Theorem} \ref{thm1212I.100-4B} appears in \hyperlink{S5544-4B}{Section} \ref{S5544-4B}, and the proof of  \hyperlink{thm1212I.100-4C}{Theorem} \ref{thm1212I.100-4C} appears in \hyperlink{S5544-4C}{Section} \ref{S5544-4C}.
\section{Integer Solutions in Short Intervals}\label{S5511-4A} \hypertarget{S5511-4A}
Standard techniques for counting the number of integers solutions of a congruence equation modulo an integer over short intervals are employed here. The case for prime moduli appears in \cite[Equation (3)]{HC1967}. A generalization to the set \textit{S}-units is provided here.  \hyperlink{lem5511.200-4C}{Lemma} \ref{lem5511.200-4C} is essentially the same as the definition of $\rho_x(m)$ given in \cite[Equation 1]{EP1953}, but much simpler.

\subsection{Integer Solutions for $T^4+1$}
Let  $f(T)=T^4+1$. The number of solutions of the congruence equation
\begin{equation}\label{eq5511.200c}
	n^4+1\equiv 0\bmod m
\end{equation}
is denoted by
\begin{equation}\label{eq5511.200e}
	\rho_f(m)=\#\left \{ n\leq m: n^4+1\equiv 0 \bmod m \right \},
\end{equation}
see \hyperlink{dfn1212W.020}{Definition} \ref{dfn1212W.060}. Let 
\begin{align}
	\mathcal{S}_0&=\{\text{prime }p: p\equiv 1\bmod 8\}.
\end{align}The subset of primes $	\mathcal{S}_0$ has positive density $\delta(\mathcal{S}_0)=1/4$ in the set of primes, see \cite[Corollary 11.19]{MV2007}, \cite[Corollary 5.29.]{IK2004} et cetera for some details on the subsets of primes representable as linear forms. The set of \textit{S}-units is defined by
\begin{equation}\label{eq5511.200a}
	\mathcal{S}(f)=\{p_1^{v_1}\cdots p_w^{v_w}>1 :p_i\equiv 1 \bmod 8 \text{ and } v_i\geq0\}.
\end{equation}
attached to the polynomial $f(T)=T^4+1$ is the support of the solution counting function $	\rho_f(m)$. The subset $\mathcal{S}(f)$ of \textit{S}-units has zero density in the set of integers since $0<\delta(\mathcal{S})<1$, it is a subset of integers quite similar to the subset of integers representable as a sum of two squares. This subset of integers is intrinsic to the analysis of the quartic polynomial $T^4+1$. 

\begin{lem} \label{lem5511.200-4A} \hypertarget{lem5511.200-4A}The congruence \eqref{eq5511.200c} is solvable over the integers if and only if $m\in 	\mathcal{S}(f)$. Furthermore, it has $\rho(m)=4^{\omega(m)}$ solutions modulo $m$.
\end{lem}
\begin{proof}[\textbf{Proof}] Let $\omega(m)=\#\{p\mid m \}$ be the number of prime divisors. For each prime $p\equiv 1 \bmod 8$ such that $p\mid m$, the finite field $\mathbb{F}_p$ contains a root $b<p/2$ of the congruence equation $n^4+1\equiv 0\bmod p$. Moreover, by Lagrange theorem for polynomials over finite fields, there are $4$ solutions $\{b,\, \zeta b,\,\zeta^2b,\,\zeta^3b\}$,  where $\zeta\in \mathbb{F}_p$ is a 4\textit{th} primitive root of unity. Since each prime divisor contributes 4 roots to the total count is $\rho(m)=4^{\omega(m)}$.
\end{proof}

The symbol $\{z\}=z-[z]$ denotes the fractional part function.
\begin{lem} \label{lem5511.200-4C} \hypertarget{lem5511.200-4C}Let $d\in	\mathcal{S}(f)$. If $x\geq 1$ is a large number, then
	\begin{equation}\label{eq5511.210d}
		\sum_{\substack{x\leq n\leq 2x\\n^4+1\equiv 0 \bmod d^2}}1=  x\frac{\rho(d^2)}{d^2}-\sum_{0\leq i\leq w-1}\left( \left\{\frac{2x-r_i}{d^2}\right\}- \left\{\frac{x-r_i}{d^2}\right\}\right),
	\end{equation}
	where $r_i\in\{0,1,2,\ldots, w-1\}$ are the roots of the congruence $n^4+1\equiv 0\bmod d^2$, and $w=\rho(d^2)=4^{\omega(m)}$.
\end{lem}

\begin{proof}[\textbf{Proof}] Now, for each root $a=r_i\in \Z /d^2\Z$ there are 
	\begin{equation}\label{eq5511.200h}
		\left[\frac{2x-a}{d^2}\right]-	\left[\frac{x-a}{d^2}\right]=\frac{x}{d^2}-\left\{\frac{2x- a}{d^2}\right\}+\left\{\frac{x- a}{d^2}\right\}
	\end{equation}
	integer solutions in the range $x\leq n=md^2+ a\leq 2x$ with $m\geq1$. Summing the contributions in \eqref{eq5511.200h} for all the roots gives the total number
	\begin{eqnarray}\label{eq5511.200k}
		N_p(x,2x)&=&\sum_{0\leq i\leq w-1}\left( \left[\frac{2x-r_i}{d^2}\right]-	\left[\frac{x-r_i}{d^2}\right]\right) \\[.3cm]
		&=&\sum_{0\leq i\leq w-1}\frac{x}{d^2}-\sum_{0\leq i\leq w-1}\left( \left\{\frac{2x-r_i}{d^2}\right\}-\left\{\frac{x-r_i}{d^2}\right\}\right) \nonumber\\[.4cm]
		&=& x\frac{\rho(d^2)}{d^2}-\sum_{0\leq i\leq w-1}\left( \left\{\frac{2x-r_i}{d^2}\right\}-\left\{\frac{x-r_i}{d^2}\right\}\right),\nonumber
	\end{eqnarray}
	of integer solutions in the range $x\leq n\leq 2x$, where $w=\rho(d^2)$ and 
\begin{equation}
\rho(d^2)=
\begin{cases}
4^{\omega(m)}&\text{ if }d\in \mathcal{S},\\
0&\text{ if }d\not \in \mathcal{S},
\end{cases}
\end{equation}	
as specified in  \hyperlink{lem5511.200-4A}{Lemma}   \ref{lem5511.200-4C}. 
\end{proof}

\begin{rmk} \label{rmk5511.200k}{\normalfont Evidently, the asymptotic approximation 
		\begin{equation}\label{eq5511.200m}
			\sum_{\substack{x\leq n\leq 2x\\n^4+1\equiv 0 \bmod p}}1= x\frac{\rho(d^2)}{d^2}+O(1)
		\end{equation}
		is simpler and more manageable, but does not work. It leads to a trivial result. However, both the exact expression and the lower bound in \eqref{eq5511.200k} lead to nontrivial results. }
\end{rmk}

\subsection{Integer Solutions for $T^4+2$}
The number of solutions of the congruence equation
\begin{equation}\label{eq5511.210c}
	n^4+2\equiv 0\bmod m
\end{equation}
is denoted by
\begin{equation}\label{eq5511.210e}
	\rho(m)=\#\left \{ n\leq m: n^4+2\equiv 0 \bmod m \right \},
\end{equation}
see \hyperlink{dfn1212W.020}{Definition} \ref{dfn1212W.060}. Let 
\begin{align}
	\mathcal{S}_0&=\{\text{prime }p=a^2+b^2: a\equiv \pm1 \text{ and } p\equiv 1\bmod 8\},\\
	\mathcal{S}_1&=\{\text{prime }p=a^2+b^2: a\equiv \pm1 \text{ and } p\equiv 3\bmod 8\},\\
	\mathcal{S}_3&=	\mathcal{S}_0\cup	\mathcal{S}_1		.
\end{align}Each of the subset of primes $	\mathcal{S}_0$ and $	\mathcal{S}_1$ has positive density in the set of primes, see \cite{IH1974} for some details on the subsets of primes representable as quadratic forms. The set of \textit{S}-units attached to the polynomial $f(T)=T^4+2$, which is defined by
\begin{equation}\label{eq5511.210a}
	\mathcal{S}(f)=\{n=p_1^{v_1}\cdots p_w^{v_w} > 1:p_i \in 	\mathcal{S}_3 \text{ and } v_i\geq0\},
\end{equation}
is the support of the solution counting function $	\rho_f(m)$. This subset of integers is intrinsic to the analysis of the quartic polynomial $T^4+2$.  \\

Set $\omega_0(m)=\#\{p\mid m: p\in 	\mathcal{S}_0\}$ and  $\omega_1(m)=\#\{p\mid m: p\in 	\mathcal{S}_1\}$.

\begin{lem} \label{lem5511.210-4A} \hypertarget{lem5511.210-4A}The congruence \eqref{eq5511.210c} is solvable over the integers if and only if $m\in 	\mathcal{S}(f)$. Furthermore, it has $\rho(m)=4^{\omega_0(m)}\times 2^{\omega_1(m)}$ solutions modulo $m$.
\end{lem}
\begin{proof}[\textbf{Proof}] Let $\omega(m)=\#\{p\mid m \}$ be the number of prime divisors. For each prime $p=a^2+b^2\equiv 1 \bmod 8$ such that $p\mid m$, the finite field $\F_p$ contains a root $r<p/2$ of the congruence equation $n^4+2\equiv 0\bmod p$. Moreover, by Lagrange theorem for polynomials over finite fields, there are $4$ solutions $\{r,\, \zeta r,\,\zeta^2r,\,\zeta^3r\}$,  and $\zeta\in \F_p$ is a 4\textit{th} primitive root of unity, see Lemma  \ref{lem1500R.800}. Likewise, for each prime $p=a^2+b^2\equiv 3 \bmod 8$ such that $p\mid m$, 
there are $2$ solutions $\{-r,\, r\}$.  
Since these are independent events, the total number of solutions is $\rho(m)=4^{\omega_0(m)}\times 2^{\omega_1(m)}$. In addition, since $2$ is not a quadratic residue whenever $p\equiv 5,7\bmod 8$ the congruence equation \eqref{eq5511.210c} has no solutions in these two cases. 
\end{proof}

The symbol $\{z\}=z-[z]$ denotes the fractional part function.
\begin{lem} \label{lem5511.210C}  \hypertarget{lem5511.210C}Let $d\in	\mathcal{S}(T^4+2)$. If $x\geq 1$ is a large number, then
	\begin{equation}\label{eq5511.210i}
		\sum_{\substack{x\leq n\leq 2x\\n^4+2\equiv 0 \bmod d^2}}1=  x\frac{\rho(d^2)}{d^2}-\sum_{0\leq i\leq w-1}\left( \left\{\frac{2x-r_i}{d^2}\right\}- \left\{\frac{x-r_i}{d^2}\right\}\right),
	\end{equation}
	where $i\in\{0,1,2,\ldots, w-1\}$ are the roots of the congruence $n^4+2\equiv 0\bmod d^2$.
\end{lem}

\begin{proof}[\textbf{Proof}] For each $d\in \mathcal{S}(f)$ the finite ring $\Z/d^2\Z$ contains $w=4^{\omega_0(m)}\times 2^{\omega_1(m)}$ roots of the congruence equation $n^4+2\equiv 0\bmod d^2$, see \hyperlink{thm1212I.100-4A}{Lemma}  \ref{lem5511.210-4A}.  Now, for each root $a=r_i\in \Z /d^2\Z$ there are 
	\begin{equation}\label{eq5511.210h}
		\left[\frac{2x-a}{d^2}\right]-	\left[\frac{x-a}{d^2}\right]=\frac{x}{d^2}-\left\{\frac{2x- a}{d^2}\right\}+\left\{\frac{x- a}{d^2}\right\}
	\end{equation}
	integer roots in the range $x\leq n=md^2+ a\leq 2x$ with $m\geq1$. Summing the contributions in \eqref{eq5511.210h} for all the roots gives the total number
	\begin{eqnarray}\label{eq5511.210k}
		N_p(x,2x)&=&\sum_{0\leq i\leq w-1}\left( \left[\frac{2x- r_i}{d^2}\right]-	\left[\frac{x-r_i}{d^2}\right]\right) \\[.3cm]
		&=&\sum_{0\leq i\leq w-1}\frac{x}{d^2}-\sum_{0\leq i\leq w-1}\left( \left\{\frac{2x-r_i}{d^2}\right\}-\left\{\frac{x-r_i}{d^2}\right\}\right) \nonumber\\[.4cm]
		&=& x\frac{\rho(d^2)}{d^2}-\sum_{0\leq i\leq w-1}\left( \left\{\frac{2x-r_i}{d^2}\right\}- \left\{\frac{x-r_i}{d^2}\right\}\right) ,\nonumber
	\end{eqnarray}
	of integer roots in the range $x\leq n\leq 2x$, where $w=\rho(d^2)$ and 
	\begin{equation}
		\rho(d^2)=
		\begin{cases}
			w=4^{\omega_0(m)}\times 2^{\omega_1(m)}&\text{ if }d\in \mathcal{S}(f),\\
			0&\text{ if }d\not \in \mathcal{S}(f),
		\end{cases}
	\end{equation}	
	as specified in  \hyperlink{lem5511.200-4A}{Lemma}   \ref{lem5511.210-4A}. 
	\end{proof}

\begin{rmk} \label{rmk5511.210k}{\normalfont Evidently, the asymptotic approximation 
		\begin{equation}\label{eq5511.210m}
			\sum_{\substack{x\leq n\leq 2x\\n^4+1\equiv 0 \bmod p}}1= x\frac{\rho(d^2)}{d^2}+O(1)
		\end{equation}
		is simpler and more manageable, but does not work. It leads to a trivial result. However, both the exact expression and the lower bound in \eqref{eq5511.210k} lead to nontrivial results. }
\end{rmk}

\subsection{Integer Solutions for $T^4-2T^2+2$}
The number of solutions of the congruence equation
\begin{equation}\label{eq5511.210-4Cc}
	n^4-n^2+2\equiv 0\bmod m
\end{equation}
is denoted by
\begin{equation}\label{eq5511.210-4Ce}
	\rho(m)=\#\left \{ n\leq m: n^4-2n^2+2\equiv 0 \bmod m \right \},
\end{equation}
see \hyperlink{dfn1212W.020}{Definition} \ref{dfn1212W.060}. Let 
\begin{align}
	\mathcal{S}_4&=\{\text{prime }p\equiv 1\bmod 8\},\\
	\mathcal{S}_5&=\{\text{prime }p\equiv 5\bmod 8\},\\
	\mathcal{S}_6&=	\mathcal{S}_4\cup	\mathcal{S}_5		.
\end{align}Each of the subset of primes $	\mathcal{S}_4$ and $	\mathcal{S}_5$ has positive density in the set of primes, see \cite[Corollary 11.19]{MV2007}, \cite[Corollary 5.29]{IK2004} et cetera for some details on the subsets of primes representable as linear forms.  The set of \textit{S}-units attached to the polynomial $f(T)=T^4-2T^2+2$, which is defined by
\begin{equation}\label{eq5511.210-4Ca}
	\mathcal{S}(f)=\{n=p_1^{v_1}\cdots p_w^{v_w} > 1:p_i \in 	\mathcal{S}_6 \text{ and } v_i\geq0\},
\end{equation}
is the support of the solution counting function $	\rho_f(m)$. This subset of integers is intrinsic to the analysis of the quartic polynomial $T^4-2T^2-2T^2+2$.  \\

Set $\omega_0(m)=\#\{p\mid m: p\in 	\mathcal{S}_0\}$ and  $\omega_1(m)=\#\{p\mid m: p\in 	\mathcal{S}_1\}$.

\begin{lem} \label{lem5511.210-4Ca} \hypertarget{lem5511.210-4Ca}The congruence \eqref{eq5511.210-4Cc} is solvable over the integers if and only if $m\in 	\mathcal{S}(f)$. Furthermore, it has $\rho(m)=4^{\omega_0(m)}\times 2^{\omega_1(m)}$ solutions modulo $m$.
\end{lem}
\begin{proof}[\textbf{Proof}] Let $\omega(m)=\#\{p\mid m \}$ be the number of prime divisors. For each prime $p\equiv 1 \bmod 8$ such that $p\mid m$, the finite field $\F_p$ contains a root $r<p/2$ of the congruence equation $n^4-2n^2+2\equiv 0\bmod p$. Moreover, by Lagrange theorem for polynomials over finite fields, there are $4$ solutions $\{r_0,\, r_1,\,r_2,\,r_3\}$. Likewise, for each prime $p\equiv 5 \bmod 8$ such that $p\mid m$, 
	there are $2$ solutions $\{-r,\, r\}$, see \cite[Corollary 2.4]{MC1980}.  
	Since these are independent events, the total number of solutions is $\rho(m)=4^{\omega_0(m)}\times 2^{\omega_1(m)}$. For the remaining primes  $p\equiv 3,7\bmod 8$ the congruence equation \eqref{eq5511.210-4Cc} has no solutions. 
\end{proof}

The symbol $\{z\}=z-[z]$ denotes the fractional part function.
\begin{lem} \label{lem5511.210-4C} \hypertarget{lem5511.210-4C}Let $d\in	\mathcal{S}(T^4-2T^2+2)$. If $x\geq 1$ is a large number, then
	\begin{equation}\label{eq5511.210-4Ci}
		\sum_{\substack{x\leq n\leq 2x\\n^4-2n^2+2\equiv 0 \bmod d^2}}1=  x\frac{\rho(d^2)}{d^2}-\sum_{0\leq i\leq w-1}\left( \left\{\frac{2x-r_i}{d^2}\right\}- \left\{\frac{x-r_i}{d^2}\right\}\right),
	\end{equation}
	where $i\in\{0,1,2,\ldots, w-1\}$ are the roots of the congruence $n^4-2n^2+2\equiv 0\bmod d^2$.
\end{lem}

\begin{proof}[\textbf{Proof}] For each $d\in \mathcal{S}(f)$ the finite ring $\Z/d^2\Z$ contains $w=4^{\omega_0(m)}\times 2^{\omega_1(m)}$ roots of the congruence equation $n^4+2\equiv 0\bmod d^2$, see \hyperlink{lem5511.210-4A}{Lemma}  \ref{lem5511.210-4A}.  Now, for each root $a=r_i\in \Z /d^2\Z$ there are 
	\begin{equation}\label{eq5511.210-4Ch}
		\left[\frac{2x-a}{d^2}\right]-	\left[\frac{x-a}{d^2}\right]=\frac{x}{d^2}-\left\{\frac{2x- a}{d^2}\right\}+\left\{\frac{x- a}{d^2}\right\}
	\end{equation}
	integer roots in the range $x\leq n=md^2+ a\leq 2x$ with $m\geq1$. Summing the contributions in \eqref{eq5511.210h} for all the roots gives the total number
	\begin{eqnarray}\label{eq5511.210-4Ck}
		N_p(x,2x)&=&\sum_{0\leq i\leq w-1}\left( \left[\frac{2x- r_i}{d^2}\right]-	\left[\frac{x-r_i}{d^2}\right]\right) \\[.3cm]
		&=&\sum_{0\leq i\leq w-1}\frac{x}{d^2}-\sum_{0\leq i\leq w-1}\left( \left\{\frac{2x-r_i}{d^2}\right\}-\left\{\frac{x-r_i}{d^2}\right\}\right) \nonumber\\[.4cm]
		&=& x\frac{\rho(d^2)}{d^2}-\sum_{0\leq i\leq w-1}\left( \left\{\frac{2x-r_i}{d^2}\right\}- \left\{\frac{x-r_i}{d^2}\right\}\right) ,\nonumber
	\end{eqnarray}
	of integer roots in the range $x\leq n\leq 2x$, where $w=\rho(d^2)$ and 
	\begin{equation}
		\rho(d^2)=
		\begin{cases}
			w=4^{\omega_0(m)}\times 2^{\omega_1(m)}&\text{ if }d\in \mathcal{S}(f),\\
			0&\text{ if }d\not \in \mathcal{S}(f),
		\end{cases}
	\end{equation}	
	as specified in  \hyperlink{lem5511.210-4Ca}{Lemma}   \ref{lem5511.210-4Ca}. 
	\end{proof}

\begin{rmk} \label{rmk5511.210-4Ck}{\normalfont Evidently, the asymptotic approximation 
		\begin{equation}\label{eq5511.210-4Cm}
			\sum_{\substack{x\leq n\leq 2x\\n^4+1\equiv 0 \bmod p}}1= x\frac{\rho(d^2)}{d^2}+O(1)
		\end{equation}
		is simpler and more manageable, but does not work. It leads to a trivial result. However, both the exact expression and the lower bound in \eqref{eq5511.210-4Ck} lead to nontrivial results. }
\end{rmk}

\section{The Main Terms}\label{S5533-4A}

\subsection{Main Term for $T^4+1$}
\begin{lem} \label{lem5511.300-4A} \hypertarget{lem5511.300-4A}If $\rho(d^2)$ denotes the number of solutions of $T^4+1\equiv 0 \bmod d^2$, then
	\begin{equation}\label{eq5511.300-4e}
		x\sum_{d^2\leq 16x^4+1}\mu(d)\frac{\rho(d^2)}{d^2}=x\prod_{p\,\equiv\, 1 \bmod 8}\left( 1-\frac{4}{p^2}\right) +O\left( \frac{1}{x^{2-\varepsilon}}\right)\nonumber
	\end{equation}
for any small number $\varepsilon>0$ as $x\to \infty$.
\end{lem}
\begin{proof}[\textbf{Proof}] Complete it as an infinite sum plus an error term and evaluate them:
\begin{eqnarray}\label{eq5511.300-4g}
x\sum_{d^2\leq 16x^4+1}\mu(d)\frac{\rho(d^2)}{d^2}
&=&x\sum_{d\geq 1}\mu(d)\frac{\rho(d^2)}{d^2}-x\sum_{d^2> 16x^4+1}\mu(d)\frac{\rho(d^2)}{d^2}\\[.3cm]
&=&x\prod_{p\,\equiv\, 1 \bmod 8}\left( 1-\frac{\rho(p^2)}{p^2}\right) +O\left( \frac{1}{x^{2-\varepsilon}}\right)\nonumber,
\end{eqnarray}
where $\rho(d^2)\ll d^{2\varepsilon}$ for any small number $\varepsilon>0$ . Last but not least, the number of solutions of the congruence modulo a prime $p$ is $\rho(p^2)=0$  or $\rho(p^2)=4$. The claim follows from these observations.
\end{proof}
\subsection{Main Term for $T^4+2$}

\begin{lem} \label{lem5511.300-4B} \hypertarget{lem5511.300-4B} If $\rho(d^2)$ denotes the number of solutions of $T^4+2\equiv 0 \bmod d^2$, then
	\begin{equation}\label{eq5511.300-4Be}
		x\sum_{d^2\leq 16x^4+2}\mu(d)\frac{\rho(d^2)}{d^2}=x\prod_{\substack{p\equiv1 \bmod 8\\p=a^2+b^2}}\left( 1-\frac{4}{p^2}\right)\prod_{\substack{p\equiv3 \bmod 8\\p=a^2+b^2}}\left( 1-\frac{2}{p^2}\right)+O\left( \frac{1}{x^{2-\varepsilon}}\right)\nonumber
	\end{equation}
	for any small number $\varepsilon>0$ as $x\to \infty$.
\end{lem}
\begin{proof}[\textbf{Proof}] Complete it as an infinite sum plus an error term and evaluate them:
	\begin{eqnarray}\label{eq5511.300-4Bg}
		x\sum_{d^2\leq 16x^4+2}\mu(d)\frac{\rho(d^2)}{d^2}
		&=&x\sum_{d\geq 1}\mu(d)\frac{\rho(d^2)}{d^2}-x\sum_{d^2> 16x^4+2}\mu(d)\frac{\rho(d^2)}{d^2}\\[.3cm]
		&=&x\prod_{\substack{p\equiv1 \bmod 8\\p=a^2+b^2}}\left( 1-\frac{4}{p^2}\right)\prod_{\substack{p\equiv3 \bmod 8\\p=a^2+b^2}}\left( 1-\frac{2}{p^2}\right)+O\left( \frac{1}{x^{2-\varepsilon}}\right)\nonumber,
	\end{eqnarray}
	where $\rho(d^2)\ll d^{2\varepsilon}$ for any small number $\varepsilon>0$ . Last but not least, the number of solutions of the congruence modulo a prime $p$ is $\rho(p^2)=0$  or $\rho(p^2)=4$. The claim follows from these observations.
\end{proof}

\subsection{Main Term for $T^4-2T^2+2$}

\begin{lem} \label{lem5511.300-4C}  \hypertarget{lem5511.300-4C}If $\rho(d^2)$ denotes the number of solutions of $T^4-2T^2+2\equiv 0 \bmod d^2$, then
	\begin{equation}\label{eq5511.300-4Ce}
		x\sum_{d^2\leq 16x^4-2x^2+2}\mu(d)\frac{\rho(d^2)}{d^2}=x\prod_{p\equiv1 \bmod 8}\left( 1-\frac{4}{p^2}\right)\prod_{p\equiv5 \bmod 8}\left( 1-\frac{2}{p^2}\right)+O\left( \frac{1}{x^{2-\varepsilon}}\right)\nonumber
	\end{equation}
	for any small number $\varepsilon>0$ as $x\to \infty$.
\end{lem}
\begin{proof}[\textbf{Proof}] Complete it as an infinite sum plus an error term and evaluate them:
	\begin{eqnarray}\label{eq5511.300-4Cg}
		x\sum_{d^2\leq 16x^4-2x^2+2}\mu(d)\frac{\rho(d^2)}{d^2}
		&=&x\sum_{d\geq 1}\mu(d)\frac{\rho(d^2)}{d^2}-x\sum_{d^2> 16x^4-2x^2+2}\mu(d)\frac{\rho(d^2)}{d^2}\\[.3cm]
		&=&x\prod_{p\equiv1 \bmod 8}\left( 1-\frac{4}{p^2}\right)\prod_{p\equiv5 \bmod 8}\left( 1-\frac{2}{p^2}\right)+O\left( \frac{1}{x^{2-\varepsilon}}\right)\nonumber,
	\end{eqnarray}
	where $\rho(d^2)\ll d^{2\varepsilon}$ for any small number $\varepsilon>0$ . Last but not least, the number of solutions of the congruence modulo a prime $p$ is $\rho(p^2)=0$  or $\rho(p^2)=4$. The claim follows from these observations.
\end{proof}
\section{The Error Terms}\label{S5533-4B}
\subsection{Error Term for $T^4+1$}
\begin{lem} \label{lem5511.350-4A}  \hypertarget{lem5511.350-4A} If $w=\rho(d^2)=4^{\omega(d^2)}$ denotes the number of solutions of $T^4+1\equiv 0 \bmod d^2$, then
	\begin{equation}\label{eq5511.350-4e}
 \sum_{0\leq i\leq  w-1,}\sum_{d^2\leq 16x^4+1}\left( \left\{\frac{2x- r_i }{d^2}\right\}+\left\{\frac{x- r_i}{d^2}\right\}\right) \ll x^{1/2+\varepsilon}	\nonumber
	\end{equation}
where $r_i$ ranges over the roots of the congruence, for any small number $\varepsilon>0$ as $x\to \infty$.
\end{lem}
\begin{proof}[\textbf{Proof}] It is sufficient to double the estimate of the finite sum of a single term. Now, observe that for 
	\begin{equation}\label{eq5511.350-4f}
	d^2\leq 2x- r_i\quad\quad \text{ and } \quad\quad 2x- r_i< d^2\leq  16x^4+1
\end{equation}	
there is an effective dyadic partition of the finite sum as	
	\begin{eqnarray}\label{eq5511.350-4h}
E(x)	&\leq&2   \sum_{0\leq i\leq  w-1,}\sum_{d^2\leq 16x^4+1} \left\{\frac{2x- r_i}{d^2}\right\}\\[.3cm]
&\leq&2  \sum_{0\leq i\leq  w-1,}\sum_{d^2\leq 2x- r_i}\left\{\frac{2x- r_i}{d^2}\right\}+2 \sum_{0\leq i\leq  w-1,}\sum_{2x- r_i<d^2\leq 16x^4+1}\left\{\frac{2x- r_i}{d^2}\right\}\nonumber\\[.3cm]
&=& E_0(x)\;+\;E_1(x)\nonumber.
	\end{eqnarray}
The first subsum has the upper bound	
\begin{eqnarray}\label{eq5511.350-4j}
E_0(x)	
&=& 2 \sum_{0\leq i\leq  w-1,}\sum_{d^2\leq 2x- r_i}\left\{\frac{2x- r_i}{d^2}\right\}\\[.3cm]
&\leq&2  \sum_{0\leq i\leq  w-1,}\sum_{d\leq 2x^{1/2}}1 \nonumber\\[.3cm]
&\leq&2\sum_{d\leq 2x^{1/2},}  \sum_{0\leq i\leq  w-1}1 \nonumber\\[.3cm]
&\ll& \sum_{d\leq 2x^{1/2}}d^{2\varepsilon}	\nonumber\\[.3cm]
&\ll& x^{1/2+\varepsilon}\nonumber,
\end{eqnarray}
\vskip .1 in 
where $\rho(d^2)=4^{\omega(d^2)}\ll d^{2\varepsilon}$ for any small number $\varepsilon>0$ . \\
	
The second subsum has the upper bound	
\begin{eqnarray}\label{eq5511.350-4l}
	E_1(x)	
	&=&2 \sum_{0\leq i\leq  w-1,}\sum_{2x- r_i<d^2\leq 16x^4+1}\left\{\frac{2x- r_i}{d^2}\right\}\\[.3cm]
	&=&2 \sum_{0\leq i\leq  w-1,}\sum_{2x- r_i<d^2\leq 16x^4+1}\frac{2x- r_i}{d^2}\nonumber\\[.3cm]
	&\leq&2\cdot 2x \sum_{0\leq i\leq  w-1,}\sum_{2x- r_i<d^2\leq 16x^4+1}\frac{1}{d^2} \nonumber\\[.3cm]
	&\leq&2\cdot 2x \sum_{2x- r_i<d^2\leq 16x^4+1}\frac{1}{d^2} \sum_{0\leq i\leq  w-1}1 \nonumber\\[.3cm]
	&\ll&x \sum_{x^{1/2}<d\leq 2x^{2}}\frac{d^{2\varepsilon}}{d^2}	\nonumber\\[.3cm]
	&\ll& x^{1/2+\varepsilon}\nonumber.
\end{eqnarray}
Summing the upper bounds of the first subsum and second subsum completes the estimate.
\end{proof}
\subsection{Error Term for  $T^4+2$}

\begin{lem} \label{lem5511.350-4B} \hypertarget{lem5511.350-4B}If $w=\rho(d^2)=4^{\omega_1(d^2)}\cdot 2^{\omega_2(d^2)}$ denotes the number of solutions of $T^4+2\equiv 0 \bmod d^2$, then
	\begin{equation}\label{eq5511.350-4Be}
		\sum_{0\leq i\leq  w-1,}\sum_{d^2\leq 16x^4+2}\left( \left\{\frac{2x- r_i }{d^2}\right\}+\left\{\frac{x- r_i}{d^2}\right\}\right) \ll x^{1/2+\varepsilon}	\nonumber
	\end{equation}
	where $r_i$ ranges over the roots of the congruence, for any small number $\varepsilon>0$ as $x\to \infty$.
\end{lem}
\begin{proof}[\textbf{Proof}] The proof is similar to the previous case in \hyperlink{lem5511.350-4A}{Lemma} \ref{lem5511.350-4A}.
\end{proof}

\subsection{Error Term for  $T^4-2T^2+2$}

\begin{lem} \label{lem5511.350-4C} \hypertarget{lem5511.350-4C}If $w=\rho(d^2)=4^{\omega_1(d^2)}\cdot 2^{\omega_2(d^2)}$ denotes the number of solutions of $T^4-2T^2+2\equiv 0 \bmod d^2$, then
	\begin{equation}\label{eq5511.350-4Ce}
		\sum_{0\leq i\leq  w-1,}\sum_{d^2\leq 16x^4-2x^2+2}\left( \left\{\frac{2x- r_i }{d^2}\right\}+\left\{\frac{x- r_i}{d^2}\right\}\right) \ll x^{1/2+\varepsilon}	\nonumber
	\end{equation}
	where $r_i$ ranges over the roots of the congruence, for any small number $\varepsilon>0$ as $x\to \infty$.
\end{lem}
\begin{proof}[\textbf{Proof}] The proof is similar to the previous case in \hyperlink{lem5511.350-4A}{Lemma} \ref{lem5511.350-4A}.
\end{proof}
\section{Main Result for $T^4+1$}\label{S5544-4A} \hypertarget{S5544-4A}
The proof of  \hyperlink{thm1212I.100-4A}{Theorem} \ref{thm1212I.100-4A} is the topic of this section. The analysis within is independent of the Galois group of the quartic polynomial. The key innovations are a new approach to evaluating the inner sum in \eqref{eq5544.400-4Af} by mean of {\color{blue}\hyperlink{lem5511.200-4C}{Lemma}  \ref{lem5511.200-4C} }and the estimation of the error term in \hyperlink{lem5511.350-4A}{Lemma} \ref{lem5511.350-4A}.

\begin{proof}[\textbf{Proof of Theorem {\normalfont \ref{thm1212I.100-4A}}}]
	
	The powerfree condition $f(T)\ne a(T)b(T)^2$ excludes the trivial case $\sum_{n\leq x,}\mu^2(f(n))=0$. Since $f(T)=T^4+1\in\Z[T]$ is irreducible over the integers and has the fixed divisor $\tdiv f=1$, the powerfree condition is satisfied. Thus, the summatory function is nontrivial
	\begin{eqnarray}\label{eq5544.400-4Af}
		N_f(x)&=&\sum_{x\leq n\leq 2x}\mu^2(n^4+1) \\
		&=&	\sum_{x\leq n\leq 2x,}\sum_{d^2\mid n^4+1}\mu(d)\nonumber\\
		&=&\sum_{d^2\leq 16x^4+1}\mu(d)\sum_{\substack{x\leq n\leq 2x\\d^2\mid n^4+1}}1\nonumber.
	\end{eqnarray}
	
Applying \hyperlink{lem5511.200-4A}{Lemma} \ref{lem5511.200-4A} to the expression \eqref{eq5544.400-4Af} returns
	\begin{eqnarray}\label{eq5544.400-4Aj}
		N_f(x)&=&\sum_{d^2\leq 16x^4+1}	\mu(d)\sum_{\substack{x \leq n\leq 2x\\n^4+1\equiv 0 \bmod d^2}}1\\[.3cm]
		&=&\sum_{d^2\leq 16x^4+1}\mu(d)\left(x\frac{\rho(d^2)}{d^2}- \sum_{\substack{0\leq i\leq 3\\1\leq j\leq w-1}}\left\{\frac{2x- r_i}{d^2}\right\}+ \sum_{\substack{0\leq i\leq 3\\1\leq j\leq w-1}}\left\{\frac{x- r_i}{d^2}\right\}\right)\nonumber\\[.3cm]
	&=&	M(x)+\,E(x)\nonumber.
	\end{eqnarray}
The main term $M(x)$ is computed in \hyperlink{lem5511.300-4A}{Lemma} \ref{lem5511.300-4A} and the error term $E(x)$ is estimated in \hyperlink{lem5511.350-4A}{Lemma} \ref{lem5511.350-4A}. Summing these expressions returns
	\begin{eqnarray}\label{eq5544.400-4Am}
	N_f(x)&=&\sum_{x\leq n\leq 2x}\mu^2(n^4+1) \\[.3cm]
		&=&	M(x)+\,E(x)\nonumber\\[.3cm]
		&=&c_fx +O\left( \frac{1}{x^{2-\varepsilon}}\right) +O( x^{1/2+\varepsilon})\nonumber\\[.3cm]
		&=&c_fx+O( x^{1/2+\varepsilon})\nonumber,
	\end{eqnarray}
where the density constant is
\begin{equation}\label{eq5544.400-4Ap}
	c_f=\prod_{p\,\equiv\, 1 \bmod 8}\left( 1-\frac{4}{p^2}\right),
\end{equation}	
and	$\varepsilon>0$ is a small real number. Quod erat demonstrandum.
\end{proof}
The approximate  numerical value of the product constant in \eqref{eq5544.400-4Ap} appears in Subsection \ref{S1212-4AD}.
\section{Main Result for $T^4+2$}\label{S5544-4B} \hypertarget{S5544-4B}
The proof of \hyperlink{thm1212I.100-4B}{Theorem} \ref{thm1212I.100-4B} is the topic of this section. The analysis within is independent of the Galois group of the quartic polynomial. The key innovations are a new approach to evaluating the inner sum in \eqref{eq5544.400-4Bf} by mean of \hyperlink{lem5511.210C}{Lemma}  \ref{lem5511.210C} and the estimation of the error term in \hyperlink{lem5511.350-4B}{Lemma} \ref{lem5511.350-4B}.

\begin{proof}[\textbf{Proof of Theorem {\normalfont \ref{thm1212I.100-4B}}}]
	
	The powerfree condition $f(T)\ne a(T)b(T)^2$ excludes the trivial case $\sum_{n\leq x,}\mu^2(f(n))=0$. Since $f(T)=T^4+2\in\Z[T]$ is irreducible over the integers and has the fixed divisor $\tdiv f=1$,  the powerfree condition is satisfied. Thus, the summatory function is nontrivial
	\begin{eqnarray}\label{eq5544.400-4Bf}
		N_f(x)&=&\sum_{x\leq n\leq 2x}\mu^2(n^4+2) \\
		&=&	\sum_{x\leq n\leq 2x,}\sum_{d^2\mid n^4+2}\mu(d)\nonumber\\
		&=&\sum_{d^2\leq 16x^4+2}\mu(d)\sum_{\substack{x\leq n\leq 2x\\d^2\mid n^4+2}}1\nonumber.
	\end{eqnarray}
	
	Applying \hyperlink{lem5511.210C}{Lemma} \ref{lem5511.210C} to the expression \eqref{eq5544.400-4Bf} returns
	\begin{eqnarray}\label{eq5544.400-4Bj}
		N_f(x)&=&\sum_{d^2\leq 16x^4+2}	\mu(d)\sum_{\substack{x \leq n\leq 2x\\n^4+2\equiv 0 \bmod d^2}}1\\[.3cm]
		&=&\sum_{d^2\leq 16x^4+2}\mu(d)\left(x\frac{\rho(d^2)}{d^2}- \sum_{\substack{0\leq i\leq 3\\1\leq j\leq w-1}}\left\{\frac{2x- r_i}{d^2}\right\}+ \sum_{\substack{0\leq i\leq 3\\1\leq j\leq w-1}}\left\{\frac{x- r_i}{d^2}\right\}\right)\nonumber\\[.3cm]
		&=&	M(x)+\,E(x)\nonumber.
	\end{eqnarray}
	The main term $M(x)$ is computed in \hyperlink{lem5511.300-4B}{Lemma} \ref{lem5511.300-4B} and the error term $E(x)$ is estimated in \hyperlink{lem5511.350-4B}{Lemma} \ref{lem5511.350-4B}. Summing these expressions returns
	\begin{eqnarray}\label{eq5544.400-4Bm}
		N_f(x)&=&\sum_{x\leq n\leq 2x}\mu^2(n^4+2) \\[.3cm]
		&=&	M(x)+\,E(x)\nonumber\\[.3cm]
		&=&c_fx+O\left( \frac{1}{x^{2-\varepsilon}}\right) +O( x^{1/2+\varepsilon})\nonumber\\[.3cm]
		&=&c_fx+O( x^{1/2+\varepsilon})\nonumber,
	\end{eqnarray}
	where the density constant is
\begin{equation}\label{eq5544.400-4Bp}
c_f=\prod_{\substack{p\equiv1 \bmod 8\\p=a^2+64b^2}}\left( 1-\frac{4}{p^2}\right)\prod_{\substack{p\equiv3 \bmod 8\\p=a^2+2b^2}}\left( 1-\frac{2}{p^2}\right),
\end{equation}	
and	$\varepsilon>0$ is a small real number. Quod erat demonstrandum.
\end{proof}

The approximate  numerical value of the product constant in \eqref{eq5544.400-4Bp} appears in Subsection \ref{S1212-4BD}.

\section{Main Result for $T^4-2T+2$}\label{S5544-4C} \hypertarget{S5544-4C}
The proof of \hyperlink{thm1212I.100-4C}{Theorem} \ref{thm1212I.100-4C} is the topic of this section. The analysis is related to the theory of modular forms. As the previous cases, the analysis it is independent of the Galois group of the quartic polynomial. The key innovations are a new approach to evaluating the inner sum in \eqref{eq5544.400-4Cf} by mean of \hyperlink{lem5511.210-4C}{Lemma}  \ref{lem5511.210-4C} and the estimation of the error term in \hyperlink{lem5511.350-4C}{Lemma} \ref{lem5511.350-4C}.

\begin{proof}[\textbf{Proof of Theorem {\normalfont \ref{thm1212I.100-4C}}}]
	
	The powerfree condition $f(T)\ne a(T)b(T)^2$ excludes the trivial case $\sum_{n\leq x,}\mu^2(f(n))=0$. Since $f(T)=T^4-2T+2\in\Z[T]$ is irreducible over the integers and has the fixed divisor $\tdiv f=1$,  the powerfree condition is satisfied. Thus, the summatory function is nontrivial
	\begin{eqnarray}\label{eq5544.400-4Cf}
		N_f(x)&=&\sum_{x\leq n\leq 2x}\mu^2(n^4-2n+2) \\
		&=&	\sum_{x\leq n\leq 2x,}\sum_{d^2\mid n^4-2n+2}\mu(d)\nonumber\\
		&=&\sum_{d^2\leq 16x^4-2x+2}\mu(d)\sum_{\substack{x\leq n\leq 2x\\d^2\mid n^4-2n+2}}1\nonumber.
	\end{eqnarray}
	
	Applying \hyperlink{lem5511.210-4C}{Lemma} \ref{lem5511.210C} to the expression \eqref{eq5544.400-4Cf} returns
	\begin{eqnarray}\label{eq5544.400-4Cj}
		N_f(x)&=&\sum_{d^2\leq 16x^4-2x^2+2}	\mu(d)\sum_{\substack{x \leq n\leq 2x\\n^4-2n+2\equiv 0 \bmod d^2}}1\\[.3cm]
		&=&\sum_{d^2\leq 16x^4-2x+2}\mu(d)\left(x\frac{\rho(d^2)}{d^2}- \sum_{\substack{0\leq i\leq 3\\1\leq j\leq w-1}}\left\{\frac{2x- r_i}{d^2}\right\}+ \sum_{\substack{0\leq i\leq 3\\1\leq j\leq w-1}}\left\{\frac{x- r_i}{d^2}\right\}\right)\nonumber\\[.3cm]
		&=&	M(x)+\,E(x)\nonumber.
	\end{eqnarray}
	The main term $M(x)$ is computed in \hyperlink{lem5511.300-4C}{Lemma} \ref{lem5511.300-4C} and the error term $E(x)$ is estimated in \hyperlink{lem5511.350-4C}{Lemma} \ref{lem5511.350-4C}. Summing these expressions returns
	\begin{eqnarray}\label{eq5544.400-4Cm}
		N_f(x)&=&\sum_{x\leq n\leq 2x}\mu^2(n^4-2n^2+2) \\[.3cm]
		&=&	M(x)+\,E(x)\nonumber\\[.3cm]
		&=&c_fx+O\left( \frac{1}{x^{2-\varepsilon}}\right) +O( x^{1/2+\varepsilon})\nonumber\\[.3cm]
		&=&c_fx+O( x^{1/2+\varepsilon})\nonumber,
	\end{eqnarray}
	where the density constant is
	\begin{equation}\label{eq5544.400-4Cp}
		c_f=\prod_{p\equiv1 \bmod 8}\left( 1-\frac{4}{p^2}\right)\prod_{p\equiv5 \bmod 8 }\left( 1-\frac{2}{p^2}\right),
	\end{equation}	
	and	$\varepsilon>0$ is a small real number. Quod erat demonstrandum.
\end{proof}

The approximate  numerical value of the product constant in \eqref{eq5544.400-4Cp} appears in Subsection \ref{S1212-4CD}.
\section{Some Numerical Data for Quartic Polynomials}\label{S1212-4D}
Let $	\mathscr{P}(q,a)=\{\text{prime }p=qn+a:\gcd(q,a)=1\}$ be a set of primes in an arithmetic progression. The 4 sets of primes in arithmetic progressions \begin{equation}
	p=8n+1, \quad p=8n+3, \quad p=8n+5 \quad \text{and} \quad p=8n+7
\end{equation} are important in the analysis of the quartic congruence equation $f(n)\equiv 0 \bmod p$. Some of these sets or subsets of these sets are used in the calculation of the natural density \begin{equation}\label{eq1212-4D-1}
c_f=\lim_{x\to\infty} \frac{\#\{n\leq x:\mu^2(f(n))\ne0\}}{x}
\end{equation}
of squarefree values of the polynomial $f(t)=t^4+a_3t^3+a_2t^2+a_1t+a_0$. Small samples of these subsets are given below.

\begin{align}\label{eq1212-4P}
	\mathscr{P}(8,1)&=\{17,   41,   73,   89,  97,  113,   137,  193,   233,  241,   257,  281,   313,   337, 353,   \\
&    \quad \quad  401,409, 433,   449,  457,   521,  569,  577,  593,  601,   617,   641, 673, 761  , \ldots.\} \nonumber\\[.5cm]
\mathscr{P}(8,3)&= \{3, 11  , 19  , 43   , 59  , 67 , , 83  , 107   , 131  , 139  , 163  , 179   , 211   , 227 , 251   , 283   , \nonumber\\
&   \quad \quad307   , 331   , 347  , 379   , 419   , 443  , 467   , 491  , 499  , 523   , 547   , 563  , 571  , 587    , \ldots.\} \nonumber\\[.5cm]
	\mathscr{P}(8,5)&= \{13,   29,  37,  53,  61,   101,  109,   149,  157,   173,  181,  197,  229,  269,     277,  \nonumber\\
& \quad \quad293, 317,   349,  373,   389,  397,   421,  461,   509,   541,  557,   613, 653  , 661, \ldots   . \}\nonumber\\[.5cm]
	\mathscr{P}(8,7)&= \{7,23  , 31   , 47   , 71  , 79  , 103   , 127   , 151   , 167   , 191  , 199   , 223   , 239   , 263  ,  271   ,\nonumber\\
& \quad \quad   311  , 359  , 367   , 383   , 431  , 439   , 463   , 479  , 487   , 503  , 599  , 607   , 631   , 647   , \ldots.\} \nonumber
\end{align}

\subsection{Data for $T^4+1$}\label{S1212-4AD}
The case $f(t)=t^4+1$ is the 8th cyclotomic polynomial. It is an irreducible polynomial with Galois group $C_4=\left( \Z/4\Z\right)^{\times} $, cyclic of order 4, see \cite[Theorem 3]{KW1989}. There are no solution for $p=2$, and the number of solution per prime $p\geq 3$ is
		\begin{equation}
			\rho_f(p)=
			\begin{cases}
				4& \text{ if } p\equiv 1 \bmod 8,\\		
				0& \text{ if } p\not \equiv 1 \bmod 8,\\
			\end{cases}
		\end{equation} 
		this follows from the cyclotomic reciprocity, \hyperlink{thm1500Y.100}{Theorem} \ref{thm1500Y.100}. The product of all the local densities $1-\rho_f(p)p^{-2}$ leads to the density constant
		\begin{eqnarray}\label{eq1212C.210k}	
			c_f&=&\prod_{p\geq 2}\left (1-\frac{\rho_f(p)}{p^2} \right )\\
			&=&	\prod_{\substack{p\geq 17\\p\equiv 1\bmod 8}}\left (1-\frac{4}{p^2} \right )\nonumber\\[.3cm]
			&\approx&0.981003777419963300927058130226506732990\ldots\nonumber.
		\end{eqnarray}
		The numerical approximation was computed using the small subset of the primes
		\begin{eqnarray}\label{eq1212C.210p}
			\mathcal{S}&=& \{17,   41,   73,   89,  97,  113,   137,  193,   233,  241,   257,  281,   313,   337,   353,   401,\quad \quad \\
			&&  409, 433,   449,  457,   521,  569,  577,  593,  601,   617,   641,  673,  761,  769,   809,     \ldots\} \nonumber.
		\end{eqnarray}
The subsets $	\mathcal{S}=	\mathscr{P}(8,1)$ coincides, see \eqref{eq1212-4P}.\\

Precisely the same analysis is valid for all the cyclotomic polynomials of the same degree. In this case there are 3 polynomials:\\
		
		$\Phi_5(T)=T^4+T^3+T^2+T+1$, \tabto{8cm} $\Phi_8(T)=T^4+1$,\\
		
		$\Phi_{12}(T)=T^4-T^2+1$.

\subsection{Data for $T^4+2$}\label{S1212-4BD}
The case $f(t)=t^4+2$ is an irreducible polynomial with Galois group $C_4=\left( \Z/4\Z\right)^{\times} $, cyclic of order 4, see \cite[Theorem 3]{KW1989}. There is one solution for $p=2$, and the number of solution per prime $p\geq 3$ is
\begin{equation}
	\rho_f(p)=
	\begin{cases}
		4& \text{ if } p\equiv 1 \bmod 8 \text{ and }p=a^2+64b^2,\\		
		2& \text{ if } p \equiv 3 \bmod 8\text{ and }p=a^2+b^2,\\		
		0& \text{ if } p \equiv 5 \bmod 8,\\
		0& \text{ if } p \equiv 7 \bmod 8,\\
	\end{cases}
\end{equation} 
this follows from the quartic reciprocity, \hyperlink{lem1500R.800}{Lemma} \ref{lem1500R.800}. The product of all the local densities $1-\rho_f(p)p^{-2}$ leads to the density constant
\begin{eqnarray}\label{eq1212C.210a}	a_f&=&\prod_{p\geq 2}\left (1-\frac{\rho_f(p)}{p^2} \right )\\
	&=&\prod_{\substack{p\equiv1 \bmod 8\\p=a^2+64b^2}}\left( 1-\frac{4}{p^2}\right)\prod_{\substack{p\equiv3 \bmod 8\\p=a^2+b^2}}\left( 1-\frac{2}{p^2}\right)\nonumber\\[.3cm]
	&\approx&0.75715941401961963321623030912\ldots\nonumber.
\end{eqnarray}

The numerical approximation was computed using the two small subsets of the primes. The first subset of primes is
\begin{eqnarray}
	\mathcal{S}_0&=& \{73,89,113,257,281,337,577,601,1033,1049,1601,1609,3137,3217,\nonumber\\
	&&4177,5209,5233,6449 
	6481,9337 ,10937,12713,16553,18617,20857 , \ldots\}\nonumber.
\end{eqnarray}
These primes are of the form $p=a^2+64b^2\equiv 1 \bmod 8$, and $2$ is a quartic residue mod $p$, see \hyperlink{lem1500R.800}{Lemma} \ref{lem1500R.800}. The subset $	\mathcal{S}_0\subset 	\mathscr{P}(8,1)$ is a proper subset of the primes in an arithmetic progression, see \eqref{eq1212-4P}. \\

 The other subset of primes are of the form $p=a^2+2b^2\equiv 3\bmod 8$ is
\begin{eqnarray}
	\mathcal{S}_1&=& \{3,11   , 19    , 43   , 59   , 67   , 83   , 107   , 131   , 139    , 163    , 179    , 211   , 227   , 251   , 283    , 307   , 331  , \nonumber\\
	&&347   , 379    , 419   , 443    , 467    , 491   , 499    , 523    , 547    , 563   , 571   , 587   , 619   , 643    , 659   , 683   , \ldots\} \nonumber,
\end{eqnarray}
The subset $	\mathcal{S}_1\subset 	\mathscr{P}(8,3)$ is a proper subset  is a proper subset of the primes in an arithmetic progression, see  \eqref{eq1212-4P}.
\\

\subsection{Data for $T^4-2T^2+2$}\label{S1212-4CD}
The case $f(T)=T^4-2T^2+2$ is irreducible polynomial with Galois group $D_4$, the dihedral group of order 8, see \cite[Theorem 3]{KW1989}. There are 2 solution for $p=2$, and the number of solutions per prime $p\geq 3$ is
		\begin{equation}
			\rho_f(p)=
			\begin{cases}
				4& \text{ if } p\equiv 1 \bmod 8,\\		
				0& \text{ if } p \equiv 3 \bmod 8,\\
				2& \text{ if } p \equiv 5 \bmod 8,\\
				0& \text{ if } p \equiv 7 \bmod 8,\\
			\end{cases}
		\end{equation} 
		the detailed proof for this case appears in \cite[Corollary 2.4]{MC1980}. Thus, the density constant is  given by
		\begin{eqnarray}\label{eq1212C.215l}
			c_f&=&\prod_{p\geq 2}\left (1-\frac{\rho_f(p)}{p^2} \right )\\
			&=&	\frac{1}{2}\prod_{\substack{p\geq 17\\p\equiv 1\bmod 8}}\left (1-\frac{4}{p^2} \right )	\prod_{\substack{p\geq 5\\p\equiv 5\bmod 8}}\left (1-\frac{2}{p^2} \right )\nonumber\\&\approx&0.9638018141353169531908731\ldots\nonumber.
		\end{eqnarray}

The subset of the primes $p=8n+1\leq 1000$ is
\begin{eqnarray}
	\mathscr{P}(8,1)&=& \{17,   41,   73,   89,  97,  113,   137,  193,   233,  241,   257,  281,   313,   337,   353,   401, \nonumber\\
	&&  409, 433,   449,  457,   521,  569,  577,  593,  601,   617,   641,  673,  761,  769,   809,   \nonumber\\
	&&857,  881,  929,  937,  953,  977,   1009,  1033,  1049,   1097,   1129,  1153,   \ldots\} \nonumber,
\end{eqnarray}
this is the same as \eqref{eq1212C.210p}. The subset of the primes $p=8n+5\leq 1000$ is
\begin{eqnarray}
	\mathscr{P}(8,5)&=& \{13,   29,  37,  53,  61,   101,  109,   149,  157,   173,  181,  197,  229,  269,  277,   293,  \nonumber\\
	&&317,   349,  373,   389,  397,   421,  461,   509,   541,  557,   613,  653,  661,  677,   701,  \nonumber\\
	&&709,   733,   757,  773,   797,  821,  829,  853, 861,  877,   941,  997,   1013,  \ldots\} \nonumber.
\end{eqnarray}

\part*{Appendix}\label{A1212}

\section{Definitions and Background Concepts } \label{S1212W}

\begin{dfn}\label{dfn1212W.060} {\normalfont The arithmetic function
\begin{equation}\label{eq1212W.060dfn}
\rho_f(p)=\#\left \{ n\leq m: f(n)\equiv 0 \bmod m \right \}\end{equation}\nonumber
tallies the number of solutions of a polynomial congruence $f(n)\equiv 0 \bmod m$, where $m\geq 2$ is an integer. 
}
\end{dfn}

The function $\rho_f(m)$ is multiplicative and can be written in the form
\begin{equation}
	\rho_f(m)=\prod_{p^b\mid\mid m}\rho_f(p^b),
\end{equation}
where $p^b\mid\mid m$ is the maximal prime power divisor of $m$, and $\rho_f(p^b)\leq \deg(f)$. Some of the properties of this function are summarized below, see \cite[p.\ 82]{RD1996}, \cite{FT1992}, et cetera, for more details.

\begin{lem} \label{lem1212W.050} The function $\rho:\N\times \Z[t] \longrightarrow \N$ satisfies the following properties.

\begin{enumerate}[font=\normalfont, label=(\roman*)]
\item $\displaystyle  \rho(n)\geq 0$ is nonnegative,
\item $\displaystyle \rho(pq)=\rho(p)\rho(q)$ is multiplicative, where $\gcd(p,q)=1$,
\item $\displaystyle \rho_f(m^2) =\prod_{p\mid m}\rho_f(p^2)\leq d^{\omega(m)}$,
\end{enumerate}
where $\deg f=d$ is the degree of the polynomial $f$, and $\omega(m)=\#\{p\mid m\}$ is the prime divisors counting function.
\end{lem}

\begin{lem} \label{lem1212W.060} Let $f(T)\in\Z[T]$ be a polynomial. Then, the roots counting function $\rho:\N\times \Z[T] \longrightarrow \N$ satisfies the following upper bounds.
	
	\begin{enumerate}[font=\normalfont, label=(\roman*)]
		\item $\displaystyle \rho_f(n^2) =\prod_{p\mid n}2 =2^{\omega(n)}\ll n^{\varepsilon},$ \tabto{7cm}if $f(T)=aT^2+bT+c$,
	\item $\displaystyle \rho_f(n^2) =\prod_{p\mid n}3 =3^{\omega(n)}\ll n^{\varepsilon},$ \tabto{7cm}if $f(T)=aT^3+bT^2+cT+d$,	
\item $\displaystyle \rho_f(n^2) =\prod_{p\mid n}4 =4^{\omega(n)}\ll n^{\varepsilon},$ \tabto{7cm}if $f(T)=aT^4+bT^3+cT^2+dT+e$,	
	\end{enumerate}
	where $\omega(n)=\#\{p\mid n\}$ is the prime divisors counting function, and $\varepsilon>0$ is a small number.
\end{lem}

\begin{proof}[\textbf{Proof}](ii) By Lagrange's theorem the polynomial $f(T)=aT^3+bT^2+cT+d$ has at most 3 roots modulo a prime power $p^2$. Thus, 
	\begin{eqnarray}\label{eq1212W.060}
		\rho(n^2) &=& \prod_{p\mid n}3 \\&=&3^{\omega(n)}\nonumber\\&\ll& 3^{2\log n/\log \log n}\nonumber\\&\ll& n^{\varepsilon}\nonumber
	\end{eqnarray}
	where $\varepsilon>0$ is a small number.  
\end{proof}


\begin{dfn}\label{dfn1212W.020}\hypertarget{dfn1212W.020} {\normalfont A \textit{separable} polynomial $f(T)\in \Z[T]$ over the integers has no repeated roots. In particular, $f(T)\ne a(T)b(T)^2$.
	}
\end{dfn}

\begin{dfn} \label{dfn2399.24} {\normalfont 
		The \textit{fixed divisor} $\tdiv(f)=\gcd(f(\mathbb{Z}))$ of a polynomial $f(x) \in \mathbb{Z}[x]$ over the integers is the greatest common divisor of its image $f(\mathbb{Z})=\{f(n):n \in \mathbb{Z}\}$.
	}
\end{dfn}
The fixed divisor $\tdiv(f)=1$ if and only if the congruence equation $f(n) \equiv 0 \mod p$ has $\rho_f(p)<p$ solutions for every prime $p<\deg(f)$, see \cite[p.\ 395]{FI2010}. An irreducible polynomial can represent infinitely many primes if and only if it has a fixed divisor $\tdiv(f)=1$.\\

\begin{exa} \normalfont A few well known polynomials are listed here.
	\begin{enumerate} 
		\item The polynomials $g_1(x)=x^2+1$ and $g_2(x)=x^2+3$ are irreducible over the integers and have the fixed divisors $\tdiv(g_1)=1$, and $\tdiv(g_2)=1$ 
		respectively. Thus, these polynomials can represent infinitely many primes. 
		\item The polynomials $g_3(x)=x(x+1)+2$ and $g_4(x)=x(x+1)(x+2)+3$ are irreducible over the integers. But, have the fixed divisors $\tdiv(g_3)=2$, and $\tdiv(g_4)=3$ respectively. Thus, these polynomials cannot represent infinitely many primes. 
		\item The polynomial $g_5(x)=x^p-x+ap$, with a prime $p\geq 2$ and an integer $a\in \Z$, is irreducible over the integers. But, has the fixed divisor $\tdiv(g_5)=p$. Thus, this polynomial cannot represent infinitely many primes. 
	\end{enumerate}
\end{exa}

\section{Quadratic and Higher Reciprocity}\label{S1500}

A \textit{reciprocity law} specifies the complete factorization of an irreducible polynomial $f(T)\in\Z[T]$ modulo a prime $p$ as the prime varies over the set of primes $\tP=\{2,3,5,\ldots \}$. Equivalently, a reciprocity rule specifies the subset of primes $\mathscr{P}_f\subset \tP$ such that the polynomial splits into linear factors over the prime finite field $\F_p$.

\subsection{Quadratic Reciprocity}\label{S1500Q}
\begin{lem} \label{lem1500Q.860}  \hypertarget{lem1500Q.860} {\normalfont (Quadratic reciprocity law)} If $p$ and $q$ are odd primes, then
	\begin{equation}\label{eq1500Q.870}
		\left ( \frac{p}{q}\right )\left ( \frac{q}{p}\right )=(-1)^{\frac{p-1}{2}\frac{q-1}{2}}.
	\end{equation}
\end{lem}
\begin{proof}[\textbf{Proof}] A detailed proof of the quadratic reciprocity laws appears in \cite[Theorem 2.1]{RH1994}, \cite{BW1998}, and similar references.
\end{proof}
\begin{lem} \label{lem1500Q.850} \hypertarget{lem1500Q.850}{\normalfont (Supplemental law)} If $p$ is an odd prime, then 
	\begin{enumerate}[font=\normalfont, label=(\roman*)]
		\item $ 2 $ is a quadratic residue if and only if $p$ is a primes of the form $p=8k\pm1$. \\
		\item $ 2 $ is a quadratic nonresidue if and only if $p$ is a primes of the form $p=8k\pm3$.
	\end{enumerate}
	Equivalently,			
	\begin{equation}\label{eq1500Q.850}
		\left ( \frac{2}{p}\right )=(-1)^{\frac{p^2-1}{8}}.
	\end{equation}
\end{lem}
\begin{proof}[\textbf{Proof}] A detailed proof of the quadratic reciprocity laws appears in \cite[Theorem 1.5.]{RH1994}, and similar references.
\end{proof}

\begin{thm}\label{thm1500Q.100} \hypertarget{thm1500Q.100}{\normalfont(Gauss)} Let $f(T)=aT^2+bT+c\in\Z[T]$ be an irreducible polynomial, and let $d=b^2-4ac\ne e^2$ be its discriminant. Then
	\begin{enumerate}[font=\normalfont, label=(\roman*)]
		\item $\displaystyle f(T)=(T-\alpha)^2, $ \tabto{6cm} if $\displaystyle p \mid d$ or $p=2$,\tabto{12cm} ramified.
		\item $\displaystyle f(T)=(T-\alpha)(T+\alpha), $ \tabto{6cm} if $\displaystyle d^{(p-1)/2}\equiv 1 \bmod p$, \tabto{12cm} split.
		\item $\displaystyle f(T)=aT^2+bT+c, $ \tabto{6cm} if $\displaystyle d^{(p-1)/2}\equiv -1 \bmod p$, \tabto{12cm} inertia .
	\end{enumerate}		
\end{thm}

\begin{proof}[\textbf{Proof}] This is derived from the values of the quadratic symbol and the Euler formula 
	\begin{equation}\label{eq1500Q.100}
		\left(\frac{d}{p} \right)=d^{(p-1)/2}\equiv \pm1 \bmod p.  
	\end{equation}
\end{proof}

\subsection{Quartic Reciprocity}\label{S1500R}
Given a prime $p=4n+1$, the quartic residue test for an arbitrary nonzero integer $u\in \F_p$ is implemented via the Euler congruence
\begin{equation}\label{eq1500R.800a}
	\left ( \frac{u}{p}\right )_4\equiv u^{\frac{p-1}{4}}\bmod p.
\end{equation}

\begin{lem} \label{lem1500R.800} \hypertarget{lem1500R.800}{\normalfont (\cite[Proportition 5.4]{LF2000})} Let $p$ be a primes of the form $p=a^2+b^2$, where $a\equiv 1 \bmod 2$. Then
	\begin{enumerate}[font=\normalfont, label=(\roman*)]
		\item  The integer $2$ is a quartic residue if and only if $p=a^2+b^2$ and $b\equiv 0\bmod 8$.  
		\item  The integer $2$ is a quartic nonresidue if $p= a^2+b^2$ and $b
		\not \equiv 0\bmod 8$.
	\end{enumerate}	 Equivalently,
	\begin{equation}\label{eq1500R.800b}
		\left ( \frac{2}{p}\right )_4=(-1)^{\frac{b}{4}}.
	\end{equation}
\end{lem}

\begin{lem} \label{lem1500R.810} \hypertarget{lem1500R.810}{\normalfont (\cite[p.\;24]{LE1958})} Let $p$ be a primes of the form $p=a^2+b^2$. Then
	\begin{enumerate}[font=\normalfont, label=(\roman*)]
		\item  The integer $3$ is a quartic residue if and only if $p=8m+1$ and $b\equiv 0\bmod 3$.  
		\item  The integer $3$ is a quartic nonresidue if and only if $p=8m+5$ and $a\equiv 0\bmod 3$.
	\end{enumerate}	
\end{lem}

The same materials on the quartic reciprocity for small integers is covered in \cite{BW1998}, \cite{IR1990}, \cite{LF2000}, et cetera.
\subsection{Cyclotomic Reciprocity}\label{S1500Y}
For any integer $n\geq1$, let $\omega$ be a primitive $n$th root of unity and let $m=\varphi(n)$. The $n$th cyclotomic polynomial is defined by
\begin{equation}\label{eq1500Y.100a}
	\Phi_n(T)=\prod_{1\leq k\leq m}(T-\omega^k).
\end{equation} 
The are various methods for computing the $n$th cyclotomic polynomial, exampli gratia, the multiplicative Mobius inversion formula yields
\begin{equation}\label{eq1500Y.100b}
	\Phi_n(T)=\prod_{d\mid n}\left(T^d-1 \right)^{\mu(n/d)}.
\end{equation}

\begin{lem} \label{lem1500Y.160}  \hypertarget{lem1500Y.160}If $n\geq1$, then  
	
	\begin{enumerate}[font=\normalfont, label=(\roman*)]
		\item $\displaystyle \Phi_n(T) $ \tabto{8cm} is irreducible over the integers.
		\item $\displaystyle \disc\left( \Phi_n(T)\right) =(-1)^{\phi(n)/2}\frac{n^{\phi(n)}}{\prod_{d\mid n}p^{\phi(n)/(p-1)}}) $ \tabto{8cm} is the discriminant.
	\end{enumerate}
\end{lem}

\begin{thm}\label{thm1500Y.100} \hypertarget{thm1500Y.100}{\normalfont(Cyclotomic reciprocity)} Let $f(T)$ be an irreducible factor of the $n$th  cyclotomic polynomial $\Phi_n(T)$ of degree $\deg f=m$ and discriminant $disc\left(\Phi_n(T) \right)$. Then
	\begin{enumerate}[font=\normalfont, label=(\roman*)]
		\item $\displaystyle f(T)=(T-\alpha)^m, $ \tabto{7cm} if $\displaystyle p \mid n $,\tabto{12cm} ramified.
		\item $\displaystyle f(T)=\prod_{1\leq i\leq m}(T-\alpha_i), $ \tabto{7cm} if $\displaystyle p-1\equiv 0 \bmod n$, \tabto{12cm} split.
		\item $\displaystyle f(T)=a_mT^m+\cdots+a_1T+a_0, $ \tabto{7cm} if $\displaystyle \ord_n p=\phi(n)$, \tabto{12cm} inertia .
	\end{enumerate}		
\end{thm}

\begin{proof}[\textbf{Proof}] (i) If $p\mid n$, then it divides the discriminant $\disc\left(\Phi_n(T)\right)$, see \hyperlink{lem1500Y.160}{Lemma} \ref{lem1500Y.160}. This implies that the polynomial is ramified.\\
	
	(ii) Given a prime $p$, and an integer $n$, let $\ord_n p=d$ be the multiplicative order of $p$ modulo $n$. The polynomial 
	\begin{equation}\label{eq1500Y.175b}
		\Phi_n(T)=f_1(T)f_2(T)\cdots f_m(T)
	\end{equation} factors into $m=\phi(n)/d$ irreducible factors $f_i(T)\in \F[T]$ of degree $\deg f_i(T)=d$. Thus, it splits into linear factors if and only if the prime $p$ has multiplicative order $\ord_n p=1$. Equivalently, $p-1\equiv 0 \bmod n$.\\
	
	(iii) Given a prime $p$, and an integer $n$, let $\ord_n p=d$ be the multiplicative order of $p$ modulo $n$. The polynomial 
	\begin{equation}\label{eq1500Y.175c}
		\Phi_n(T)=f_1(T)f_2(T)\cdots f_m(T)
	\end{equation} factors into $m=\phi(n)/d$ irreducible factors $f_i(T)\in \F[T]$ of degree $\deg f_i(T)=d$. Thus, it remains irreducible if and only if the prime $p$ has multiplicative order $\ord_n p=\phi(n)$.
\end{proof}

\section{Problems}\label{EXE1500}

\subsection{Factorization and Reciprocity Problems}\label{EXE1505A}

\begin{exe}\label{exe1505A.010a} {\normalfont Use the quadratic reciprocity law to determine the factorization of the polynomial $f(T)=T^2+2$.\\
		
	}
\end{exe}

\begin{exe}\label{exe1505A.010b} {\normalfont Use the cubic reciprocity law to determine the factorization of the polynomial $f(T)=T^3+2$.\\
		
	}
\end{exe}

\begin{exe}\label{exe1505A.010c} {\normalfont Use the quartic reciprocity law to determine the factorization of the polynomial $f(T)=T^4+2$.\\
	}
\end{exe}

\begin{exe}\label{exe1505A.010k} {\normalfont Apply the Mobius inversion formula to show that the \textit{n}-cyclotomic polynomial has the factorization	
$$
	\Phi_n(T)=(x^n-1)\prod_{1<d\mid n}\Phi_d(T)^{-1}.
$$
}
\end{exe}

\subsection{Densities of Subsets of Squarefree Integers}\label{EXE1505B}
\begin{exe}\label{exe1505B.010a} {\normalfont Show that the value set $\mathscr{V}_f=\{n:f(n) \text{ is squarefree}\}\subset \N$ of linear polynomial $f(T)=aT+b$ has the minimal density $c_f=6/\pi^2p$ in the set of integers $\N=\{0,1,2,3, \ldots\}$ if and only if $a=1$.\\
		
	}
\end{exe}

\begin{exe}\label{exe1505B.010b} {\normalfont Show that the value set $\mathscr{V}_f=\{n:f(n) \text{ is squarefree}\}\subset \N$ of linear polynomial $f(T)=aT+b$ cannot has the maximal density $c_f=1$ in the set of integers $\N=\{0,1,2,3, \ldots\}$ for any integer $a\ne0$. Specifically, $\mathscr{V}_f\ne \N$ for any linear polynomial.\\
		
	}
\end{exe}

\begin{exe}\label{exe1505B.110m} {\normalfont Verify that the subset $\mathcal{Q}_o=\{n\geq1:\mu(n)\ne0 \text{ and } n\equiv 1 \bmod 2\}$ of odd squarefree integers has the density constant
		$\displaystyle  c_o= \prod_{p\geq3}\left( 1-\frac{1}{p^2}\right)=\frac{8}{\pi^2}$
		in the set of natural numbers $\N$.
	}
\end{exe}

\subsection{Multiplicative Orders}\label{EXE1505F}
\begin{exe}\label{exe1505B.010m} {\normalfont Let $u\ne \pm1,t^2$ be a primitive root modulo $p$, and let $\gcd(m,p-1)=1$. Show that the consecutive powers $1+u+u^2+\cdots+u^{m-1}=v$ is a primitive root modulo $p$. Hint: use the geometric series and $(1-v)/(1-u)$.
	}
\end{exe}

\subsection{Characteristic Functions}\label{EXE1590CF}

\begin{exe}\label{exe1590CF.010f} {\normalfont Let $v_p(n)=\max\{v:p^v\mid n\}$ be the $p$-adic valuation, and let $\mu^2(n)$ be the characteristic function of squarefree rational integers. Verify the followings statements. \\
		
(a) $\displaystyle  \kappa(p)=\frac{1}{p}\sum_{0\leq a<p}e^{i2 \pi an/p^2}=\begin{cases}
0&\text{ if }	v_p(n)\leq 1,\\
1&\text{ if }	v_p(n)\geq 2.\
		\end{cases}$\\
		
(b)	 $ \displaystyle\mu^2(n)=\prod_{p\mid n}\left( 1-\kappa(p)\right) $.
	}
\end{exe}

\subsection{Generating Functions of  Squarefree Integers}\label{EXE1570F}
\begin{exe}\label{exe1505D.150m} {\normalfont Verify that the subset of odd squarefree integers $\mathcal{Q}_o=\{n\geq1:\mu(n)\ne0 \text{ and } n\equiv 1 \bmod 2\}$ has the generating function  $$  \sum_{n\geq 1} \frac{\left( 1+(-1)^n\right) \mu^2(n)}{n^s}=c(s,2,1)\frac{\zeta(s)}{\zeta(2s)},$$
where $c(s,2,1)$ is a rational correction factor and $s\in \C$ is a complex number.
	}
\end{exe}

\begin{exe}\label{exe1505D.150p} {\normalfont Verify that the subset of squarefree integers $\mathcal{Q}(4,1)=\{n\geq1:\mu(n)\ne0 \text{ and } n\equiv 1 \bmod 4\}$  has the generating function  $$  \sum_{n\geq 1} \frac{\left( 1+\chi(n)\right) \mu^2(n)}{n^s}=c(s,4,1)\frac{\zeta(s)}{\zeta(2s)}$$
		where $\chi(n)=(n | 4)$ is the quadratic symbol, $c(s,4,1)$ is an irrational correction factor and $s\in \C$ is a complex number.
	}
\end{exe}

\begin{exe}\label{exe1505D.150r} {\normalfont Verify that the subset of squarefree integers  $\mathcal{Q}(4,1)=\{n\geq1:\mu(n)\ne0 \text{ and } n\equiv 3 \bmod 4\}$ has the generating function  $$  \sum_{n\geq 1} \frac{\left( 1-\chi(n)\right) \mu^2(n)}{n^s}=c(s,4,3)\frac{\zeta(s)}{\zeta(2s)}$$
		where $\chi(n)=(n | 4)$ is the quadratic symbol, $c(s,4,3)$ is an irrational correction factor and $s\in \C$ is a complex number.
	}
\end{exe}
\subsection{Numerical Computations}\label{EXE1550F}
\begin{exe}\label{exe1505B.150m} {\normalfont Determine the values of the following products to 100 decimal places accuracies.
\begin{multicols}{2}
\begin{enumerate}
\item[(a)] $\displaystyle  A= \prod_{\substack{p\equiv1 \bmod 8\\p=a^2+64b^2}}\left( 1-\frac{4}{p^2}\right)$	,
\item[(b)] $\displaystyle B=  \prod_{\substack{p\,\equiv\, 3 \bmod 8\\p=a^2+b^2}}\left( 1-\frac{2}{p^2}\right),$	
\end{enumerate}
\end{multicols}
and the density constant $c_f=AB$ associated with the squarefree values of the polynomial $f(t)=t^4+2$.
	}
\end{exe}

\subsection{Proofs for other Quartic Polynomials}\label{EXE1509A}

\begin{exe}\label{exe1509A.005a} {\normalfont Use \hyperlink{thm1212I.100-4A}{Lemma} \ref{lem1500R.810} to prove that the polynomial $f(T)=T^4+3$ has infinitely many squarefree values. More precisely, prove the asymptotic formula$$
		\sum_{x\leq n\leq 2x}\mu^2(n^4+3)=	c_fx+O( x^{1/2+\varepsilon}),$$
		where $c_f>0$ is a constant, and $\varepsilon>0$ is arbitrary.
	}
\end{exe}

\subsection{Smooth Values of Polynomials}\label{EXE2020S}
\begin{exe}\label{exe2020S.010m} {\normalfont  Let $f(t)=t^2+1$ and let $B>1$ be a small real number. Find an asymptotic formula for the the cardinality of the subset $\{n\leq x:p\mid f(n) \Rightarrow p\leq B\}$ smooth value of the polynomial as $x\to\infty.$
	}
\end{exe}

\subsection{Open Problems}\label{EXE1515}
\begin{exe}\label{exe1515.010a} {\normalfont Let $f(T)=T^2+1$. Since the integer values of a polynomial are not consecutive, any consecutive pattern of $k$-tuples are seems to be possible.  Does the pattern of $k$-consecutive constant values such as $$\left( \mu(n^2+1), \mu((n+1)^2+1), \mu((n+2)^2+1),\ldots,\mu((n+k-1)^2+1)\right)=(1,1,1,\ldots,1) $$ occur infinitely often as $n\to\infty$?\\
		
	}
\end{exe}

\begin{exe}\label{exe1515.010b} {\normalfont Let $f(T)=T^4+2$. Since the integer values of a polynomial are not consecutive, any consecutive pattern of 4-tuples are possible.  Does the pattern of $k$-consecutive constant values such as $$\left( \mu(n^4+2), \mu((n+1)^4+2), \mu((n+2)^4+2),\mu((n+3)^4+2)\right)=(1,1,1,1) $$ occur infinitely often as $n\to\infty$?
}
\end{exe}



\end{document}